\documentclass{article}
\usepackage{amsmath,amssymb,amsthm}
\newtheorem{thm}{Theorem}[section]
\newtheorem{pro}[thm]{Proposition}
\newtheorem{lem}[thm]{Lemma}

\newtheorem{cor}[thm]{Corollary}

\theoremstyle{definition}

\newtheorem{defi}[thm]{Definition}

\begin{document}
\date{}
\title{\bf New code upper bounds  for the folded $n$-cube}
\author{Lihang Hou${}^1$\, \ Bo Hou${}^1$\, \ Suogang Gao${}^1$\thanks{Corresponding author. E-mail address: sggaomail@163.com.}\, \ Wei-Hsuan Yu${}^2$\\
{\footnotesize${}^1$ College of Mathematics and Information Science,  Hebei Normal University,
Shijiazhuang,  050024, P. R. China}\\
{\footnotesize${}^2$ Institute for Computational and Experimental Research in Mathematics, Brown University,  Providence, RI, 02903, USA}}
\maketitle
\begin{abstract}
Let $\Gamma$ denote a distance-regular graph. The maximum size of codewords with minimum distance at least $d$ is denoted by $A(\Gamma,d)$. Let $\square_n$ denote the folded $n$-cube $H(n,2)$. We give an upper bound on $A(\square_n,d)$ based on  block-diagonalizing the Terwilliger
algebra of $\square_n$ and on semidefinite programming.
The technique  of this paper is an extension of the
approach taken by A. Schrijver \cite{s} on the study of $A(H(n,2),d)$.
\end{abstract}
\section{Introduction}
Let $\Gamma$ denote a distance-regular graph with vertex set $V\Gamma$, path-length distance function $\partial$ and diameter $D$. We call  any nonempty subset $C$ of $V\Gamma$ a code in $\Gamma$. For $1<|C|<|V\Gamma|$, the minimum distance
of $C$ is defined as $d:=\text{min}\{\partial(x,y)|x,y\in C,\ x\neq y\}$.  The maximum size of $C$ with minimum distance at least $d$ is denoted by $A(\Gamma,d)$. In general, the problem of determining $A(\Gamma,d)$ is difficult and hence any improved upper bounds are interesting enough for the researchers in this area. In \cite{s},  A. Schrijver introduced
 a new method based on block-diagonalizing the Terwilliger algebra of $H(n,2)$ and on semidefinite programming to give an upper bound on $A(H(n,2),d)$. This  method can be seen
as a refinement of Delsarte's linear programming approach \cite{d} and
the obtained new bound  is stronger than  the Delsarte  bound.
 In \cite{g1} these results were extended to the $q$-Hamming scheme  with $q\geq 3$. We refer the reader to \cite{g} for more details on this method.

Motivated by above works, in this paper we will consider the folded $n$-cube $H(n,2)$ which is  denoted  by $\square_n$. We first determine the Terwilliger algebra of $\square_n$ with respect to a fixed vertex. Then based on block-diagonalizing the Terwilliger algebra of $\square_n$  and on semidefinite programming, we give a new upper bound on $A(\square_n,d)$. This  bound strengthens the Delsarte
bound and  can be calculated in time polynomial in $n$ using semidefinite programming.

We now recall the definition of $\square_n$.
Let $S=\{1,2,\ldots, n\}$ with integer $n\geq 6$. It is known that each subset of $S$ is called the $support$ of  vertex of $H(n,2)$ and hence we can identify all  vertices of $H(n,2)$  with their support. Then the $Hamming\ distance$ of $u,v\subseteq S$ is equal to $|u\triangle v|$, where $u\triangle v=u\cup v-u\cap v$.
Denote by $X$  the set of all unordered pairs $(u,u')$,
where $u,u'\subseteq S$, $u\cap u'=\emptyset$, $u\cup u'=S$.
 $\square_n$ can be described as the graph
whose vertex set is $X$, two vertices, say $z:=(z_1,z_2),w:=(w_1,w_2)$, are adjacent whenever
$\text{min}\{|z_i\triangle w_j|: i,j=1,2\}=1$. Thus the path-length distance of $x:=(x_1,x_2)$ and $y:=(y_1,y_2)$ is given by
\begin{align}
\partial(x,y)=\text{min}\{|x_i\triangle y_j|: i,j=1,2\}.\nonumber
\end{align}
Observe that
$|x_1\triangle y_1|=|x_2\triangle y_2|$, $|x_1\triangle y_2|=|x_2\triangle y_1|$, and $|x_1\triangle y_1|+|x_1\triangle y_2|=n$. Then it follows that  $\partial(x,y)=\text{min}$$\{|x_1\triangle y_1|$,$|x_1\triangle y_2|\}$ and
$0\leq \partial(x,y)\leq \lfloor \frac{n}{2}\rfloor$, where $\lfloor a\rfloor$ denotes the maximal integer less than or equal to $a$.  It is well-known that $\square_n$ is a $bipartite$ (an $almost$-$bipartite$) distance-regular graph with diameter $\lfloor \frac{n}{2}\rfloor$  for even $n$ (odd $n$).

The paper is organized as follows.
In Section $2$, we recall some definitions and  facts concerning
the distance-regular graph and its Terwilliger algebra.  In Section $3$,  we give a basis of the Terwilliger algebra of $\square_n$ by considering the action of automorphism group of $\square_n$ on $X\times X\times X$. In Section $4$, we study a block-diagonalization of the Terwilliger algebra  via the obtained basis. In Section $5$, we estimate an upper bound on $A(\square_n,d)$ by semidefinite programming involving the block-diagonalization  of  the Terwilliger algebra. Moreover, we offer several concrete  upper bounds on $A(\square_n,d)$ for $8\leq n\leq 13$.
\section{Preliminaries}
Let $\Gamma$ denote a  distance-regular graph with vertex set $V\Gamma$, path-length distance function $\partial$, and diameter $D$. Let $V=\mathbb{C}^{V\Gamma}$ denote the $\mathbb{C}$-space of column vectors  with coordinates indexed by $V\Gamma$, and
let ${\rm{Mat}}_{V\Gamma}(\mathbb{C})$ denote the $\mathbb{C}$-algebra of matrices with rows and columns indexed by $V\Gamma$.

For $0\leq i\leq D$ let $A_i\in {\rm{Mat}}_{V\Gamma}(\mathbb{C})$  denote the $i$th {\it{distance matrix}} of $\Gamma$: $A_i$ has
$(x,y)$-entry equal to $1$ if $\partial(x,y)=i$ and $0$ otherwise.
 It is known that $A_0, A_1, \ldots, A_D$ span
a commutative subalgebra  of  ${\rm{Mat}}_{V\Gamma}(\mathbb{C})$, denoted by $\mathcal{M}$. It turns out that $\mathcal{M}$ can be generated by $A_1$.
We call $\mathcal{M}$ the
{\it{Bose-Mesner algebra}} of $\Gamma$. Fix a vertex $x\in V\Gamma$. For $0\leq i\leq D$ let diagonal matrix $E^*_i=E^*_i(x)$ denote $i$th {\it{dual idempotent}} of $\Gamma$: $E^*_i$ has
$(y,y)$-entry equal to $1$ if $\partial(x,y)=i$ and $0$ otherwise.
It is known that $E^*_0, E^*_1, \ldots, E^*_D$ span  a commutative subalgebra of ${\rm{Mat}}_X(\mathbb{C})$, denoted by $\mathcal{M}^*$. We call $\mathcal{M}^*$ the
{\it{dual Bose-Mesner algebra}} of $\Gamma$ with respect to $x$.

Let $T=T(x)$ denote the subalgebra of ${\rm{Mat}}_{V\Gamma}(\mathbb{C})$
generated by $\mathcal{M}$ and $\mathcal{M}^*$, and $T$ is called
the {\it Terwilliger algebra} of $\Gamma$ with respect to $x$. It is known that $T$ is semisimple and finite dimensional.
In what follows,
we  recall some terms about $T$-modules. A subspace $W\subseteq V$ is called  $T$-{\it{module}} if $YW\subseteq W$
for all $Y\in T$.  $W$ is said to be $irreducible$ whenever $W\neq 0$ and $W$ contains no $T$-modules besides $0$ and $W$. Assume $W$ is an irreducible $T$-module. By the {\it{endpoint}} of $W$ (resp. {\it{diameter}} of $W$), we
mean ${\rm{min}} \{i | 0\leq i\leq D, E^*_iW\neq 0\}$ (resp. $|\{i | 0\leq i\leq D, E_i^*W\neq 0\}|-1$).
$W$ is said to be $thin$ whenever $dim(E^*_iW)\leq 1$ for all $0\leq i\leq D$. Note that the standard module $V$ is an orthogonal direct sum of irreducible $T$-modules. By the multiplicity with
which $W$ appears in $V$, we mean the number of irreducible $T$-modules in this sum which
are isomorphic to $W$.
See \cite{cau,ca,p2,ter3} for more information on the Terwilliger algebra.
\begin{lem}{\rm (\cite[Lemma 3.9]{p2})}\label{lem1} Let $W$ denote an irreducible $T$-module with endpoint
 $r$ and diameter $d^*$. Then the following {\rm{(i)}--\rm{(iii)}} hold.
\begin{itemize}
\item[\rm(i)] $A_1E^*_iW\subseteq E^*_{i-1}W+E^*_iW+E^*_{i+1}W$\ $(0\leq i\leq D)$.
\item[\rm(ii)] $E^*_iW\neq 0$ if and only if $r\leq i\leq r+d^*$.
\item[\rm(iii)] $E^*_jA_1E^*_iW\neq 0$ if $|j-i|=1$\ \ $(r\leq i,j \leq d^*)$.
\end{itemize}
\end{lem}
\begin{lem}\label{lem2} Let $W$ denote a thin irreducible $T$-module with endpoint
 $r$ and diameter $d^*$. Pick a nonzero vector $\xi_0\in E^*_rW$, and let
$\xi_i=E^*_{r+i}A_1E^*_{r+i-1}A_1E^*_{r+i-2}\cdots E^*_{r+1}A_1E^*_{r}\xi_0$ $(1\leq i\leq d^*)$.
Then we have $\xi_i\in E^*_{r+i}W$ and $\xi_i$ is nonzero. Moreover,
$\xi_0,\xi_1,\ldots,\xi_{d^*}$ span $W$.
\end{lem}
\begin{proof}
It is easy to see that $\xi_i\in E^*_{r+i}W$. Since $W$ is thin, we have dim$(E^*_iW)=1$ for
$r\leq i\leq r+d^*$ by Lemma \ref{lem1}(ii). Then use Lemma \ref{lem1}(iii) to induct on $i$. We can have that each $\xi_i\ (1\leq i\leq d^*)$ is nonzero and hence $\xi_0,\xi_1,\ldots,\xi_{d^*}$ are linearly independent. It follows from $dim(W)=d^*+1$ that $W=span\{\xi_0,\xi_1,\ldots,\xi_{d^*}\}$.
\end{proof}
At end of this section, we recall some facts from number theory which are useful later.
\begin{lem}\label{lem 2.8} The following {\rm (i)--(iii)} hold.
\begin{itemize}
\item[\rm(i)] The number of nonnegative integer solutions to the equation  $x_1+x_2+\cdots +x_m=n$ is  ${n+m-1\choose m-1}$.
\item[\rm(ii)] $\sum_{k=0}^{n}(-1)^{k-m}{k\choose m}{n\choose k}=\delta_{m,n}$.
\item[\rm(iii)]  $\sum^m_{k=0}(-1)^{m-k}{m\choose k}{n-2m+k\choose n-i}={n-2m\choose i-m}$.
\end{itemize}
\end{lem}
\section{The Terwilliger algebra of $\square_n$}
In this section, we give a basis of  the Terwilliger algebra of $\square_{n}$ with $n\geq 6$.
We treat two cases of $n$ even and odd separately.
\subsection{The Terwilliger algebra of $\square_{2D}$}
Recall the definition of  vertex set $X$ for $n=2D$ and  we can view $X$ as the  set  consisting of all ordered pairs $(u,u')$ with $|u|<|u'|$ and all unordered pairs $(u,u')$ with $|u|=|u'|$.
We  give  the following notation.
To each ordered triple $(x,y,z)\in X\times X\times X$, where
 $x=(x_1,x_2), y=(y_1,y_2), z=(z_1,z_2)$,
we associate the integers three-tuple $(i,j,t)$:
\begin{align*}
\partial(x,y,z):=(i,j,t),\ \ \ \ \ \ \ \ \ \  \text{where}\ \ i&:=\partial(x,y),\ \ \ \ \ \ \ \
\ \ \ \ \ \ \ \  \ \ \ \ \  \nonumber\\
 j&:=\partial(x,z),
\end{align*}
without loss of generality, let $|x_1\bigtriangleup y_1|=i$ and $|x_1\bigtriangleup z_1|=j$. Then
\begin{align}
&\text{for}\ 0\leq i,j\leq D-1,\  \ t:=|(x_1\bigtriangleup y_1) \cap  (x_1\bigtriangleup z_1)|,\nonumber \\
&\text{for}\ i=D,0\leq j\leq D-1,\ \ t:=\text{max}\{|(x_1\bigtriangleup y_1)\cap  (x_1\bigtriangleup z_1)|, |(x_1\bigtriangleup y_2)\cap  (x_1\bigtriangleup z_1)|\},\nonumber\\
&\text{for}\ 0\leq i\leq D-1,j=D,\ \ t:=\text{max}\{|(x_1\bigtriangleup y_1)\cap  (x_1\bigtriangleup z_1)|, |(x_1\bigtriangleup y_1)\cap  (x_1\bigtriangleup z_2)|\},\nonumber\\
&\text{for}\ i=j=D,\ \ t:=\text{max}\{|(x_1\bigtriangleup y_1)\cap  (x_1\bigtriangleup z_1)|, |(x_1\bigtriangleup y_1)\cap  (x_1\bigtriangleup z_2)|,\nonumber\\
 & \ \ \ \ \ \ \ \ \ \ \ \ \ \ \ \ \ \ \ \ \ \ \ \ \ \ \ \ \ \ \ \ \ \ |(x_1\bigtriangleup y_2)\cap  (x_1\bigtriangleup z_1)|, |(x_1\bigtriangleup y_2)\cap  (x_1\bigtriangleup z_2)|\}\nonumber\\
 &\ \ \ \ \ \ \ \ \ \ \ \ \ \ \ \ \ \ \ \ \ \ \  =\text{max}\{|(x_1\bigtriangleup y_1)\cap  (x_1\bigtriangleup z_1)|, |(x_1\bigtriangleup y_1)\cap  (x_1\bigtriangleup z_2)|\}.\nonumber
\end{align}
Observe that $0\leq t\leq i,j\leq D$, $t\geq \lfloor\frac{j+1}{2}\rfloor$ for $i=D$, and  $t\geq \lfloor\frac{i+1}{2}\rfloor$ for $j=D$. Note that $\partial(y,z)=\text{min}\{|y_1\bigtriangleup z_1|=|y_2\bigtriangleup z_2|, |y_1\bigtriangleup z_2|=|y_2\bigtriangleup z_1|\}.$ Then by  simple calculation, we have that $\partial(y,z)=\text{min}\{i+j-2t,\  2D-(i+j-2t)\}$ for $0\leq i,j\leq D-1$ and $\partial(y,z)=i+j-2t$ for $i=D$ or $j=D$.
The set of three-tuples
$(i,j,t)$ that occur as $\partial(x,y,z)=(i,j,t)$ for some $x,y,z\in X$ is given by
\begin{align}
\mathcal{I}:=\{(i,j,t)|\  &0\leq t\leq i,j\leq D,\  i+j-t\leq 2D-2, \nonumber\\
& t\geq \lfloor\frac{j+1}{2}\rfloor\ \text{if\ $i=D$ and}\ t\geq \lfloor\frac{i+1}{2}\rfloor\ \text{if\ $j=D$}\}. \label{eq36}
\end{align}
\begin{pro}\label{pro 3}
We have
\begin{equation}\nonumber
|\mathcal{I}|=\frac{(D+1)(D^2+2D+3)}{3}.
\end{equation}
\end{pro}
\begin{proof}
Let
\begin{equation}\label{eq35}
i+j-t=l\ \ (0\leq t\leq i,j\leq D,\ 0\leq l\leq 2D-2).
\end{equation}
We divide the proof into  three cases.\\
(i) the case: $0\leq l\leq D$. Substitute $i':=i-t$ and $j':=i-t$. Then the integer solutions of \eqref{eq35} are in bijection with the integer solutions of
\begin{equation}\label{eq6}
0\leq i',j', t\leq D,\  i'+j'+t= l.
\end{equation}
By Lemma \ref{lem 2.8}(i) the number of integer solutions of \eqref{eq6} is ${l+2\choose 2}$ and these solutions satisfy \eqref{eq36}.\\
(ii) the case: $D+1\leq l\leq D+\lfloor\frac{D}{2}\rfloor$. Substitute $i':=D-i$, $j':=D-j$ and $l':=2D-l$ Then the integer solutions of \eqref{eq35}
are in bijection with the integer solutions of
\begin{equation}\label{eq8}
0\leq i',j', t\leq D,\  i'+j'+t= l'.
\end{equation}
The number of integer solutions of \eqref{eq8} is ${l'+2\choose 2}={2D-l+2\choose 2}$. One easily verifies that when $i=D$ or $j=D$ in \eqref{eq35} there are total $2(l-D)$ integer solutions satisfying \eqref{eq8} but not satisfying \eqref{eq36}.\\
(iii) the case: $D+\lfloor\frac{D}{2}\rfloor+1\leq l\leq 2D-2$. By the argument similar to the discussion  of case (ii), we have that
the number of integer solutions satisfying \eqref{eq36} is  ${2D-l+2\choose 2}-2(2D-l)-1={2D-l\choose 2}$. Note that  when $i=D$ or $j=D$ in \eqref{eq35} there are total $2(2D-l)+1$ integer solutions not satisfying \eqref{eq36}.

Therefore,
\begin{align}
|\mathcal{I}|&=\sum^D_{l=0}{l+2\choose 2}+\sum^{D+\lfloor\frac{D}{2}\rfloor}_{l=D+1}\bigg({2D-l+2\choose 2}-2(l-D)\bigg)+\sum^{2D-2}_{l=D+\lfloor\frac{D}{2}\rfloor+1}{2D-l\choose 2}\nonumber\\
&=\frac{(D+1)(D+2)(D+3)}{6}+\frac{D(D+1)(D+2)}{6}-\frac{(D-\lfloor\frac{D}{2}\rfloor)
(D-\lfloor\frac{D}{2}\rfloor+1)(D-\lfloor\frac{D}{2}\rfloor+2)}{6}\nonumber\\
&\ \ \ \ \ \ \ \ \ \ \ \ \ \ \ \ \ \ \ \ \ \ \ \ \ \ \ \ \ \ \ \ \ \ \  -\lfloor\frac{D}{2}\rfloor(\lfloor\frac{D}{2}\rfloor+1)+\frac{(D-\lfloor\frac{D}{2}\rfloor-2)
(D-\lfloor\frac{D}{2}\rfloor-1)(D-\lfloor\frac{D}{2}\rfloor)}{6}\nonumber\\
&=\frac{(D+1)(D^2+2D+3)}{3}\nonumber.
\end{align}
\end{proof}
For each $(i,j,t)\in \mathcal{I}$, we  define
\begin{align}
X_{i,j,t}:=\{(x,y,z)\in \{X\times X\times X| \partial(x,y,z)=(i,j,t)\}.
\end{align}
Denote  by {\rm Aut}$(X)$ the automorphism group of $\square_{2D}$ and ${\rm Aut}_{\textbf{0}}(X)$ the stabilizer of vertex $\textbf{0}:=(\emptyset,S)$ in ${\rm Aut}(X)$.
The following proposition gives the meaning of $X_{i,j,t}$, $(i,j,t)\in \mathcal{I}$.
\begin{pro}\label{pro 15}
The sets $X_{i,j,t}$, $(i,j,t)\in \mathcal{I}$ are the orbits
of $X\times X\times X$ under the action
of {\rm Aut}$(X)$.
\end{pro}
\begin{proof}
By \cite[p. 265]{bcn} the {\rm Aut}$(X)$  is $2^{2D-1}$.sym$(2D)$. Let $x,y,z\in X$ and let $\partial(x,y,z)=(i,j,t)$. By the definitions of $i,j$ and $t$, one easily  verifies  that
$i,j,t$ are unchanged under any action of  $\sigma \in {\rm Aut}(X)$, that is $\partial(\sigma x,\sigma y,\sigma z)=(i,j,t)$.

To show that {\rm Aut}$(X)$ acts transitively on $X_{i,j,t}$ for each $(i,j,t)\in \mathcal{I}$, it suffices to show that for fixed $\partial(x',y',z')=(i,j,t)$ if $\sigma\in {\rm Aut}(X)$  ranges over ${\rm Aut}(X)$ then $(\sigma x',\sigma y',\sigma z')$ ranges over $X_{i,j,t}$ . By permuting on $X$, we may assume that $x'=\textbf{0}$. Then  $\partial(\textbf{0},y',z')=(i,j,t)$. Since ${\rm Aut}_{\textbf{0}}(X)$ is sym$(2D)$, we have that if $\psi\in {\rm Aut}_{\textbf{0}}(X)$  ranges over the ${\rm Aut}_{\textbf{0}}(X)$ then $(\psi y',\psi z')$ ranges over
the set $\{(y,z)\in X\times X|\partial(\textbf{0},y,z)=(i,j,t)\}$.
\end{proof}

The action of ${\rm Aut}(X)$ on  $X\times X\times X$
induces an action of ${\rm Aut}_{\textbf{0}}(X)$ on $\{\textbf{0}\}\times X\times X$. Thus we define
\begin{align}
 X^\textbf{0}_{i,j,t}:=\{(x,y)\in X\times X| \partial(\textbf{0},x,y)=(i,j,t)\}.\nonumber
\end{align}

Observe that  $(x,y)\in X^\textbf{0}_{i,j,t}$ is equivalent to  $|x_1|=i,|y_1|=j$ and\\
\text{\ \ \ \ \ \ \ \ \ \ \ \ }$t=|x_1\cap  y_1|$ when \ $0\leq i,j\leq D-1$,\   \\
\text{\ \ \ \ \ \ \ \ \ \ \ \ }$t=\text{max}\{|x_1\cap  y_1|, |x_2\cap  y_1|\}$ when\ $i=D,\ 0\leq j\leq D-1$,\ \\
\text{\ \ \ \ \ \ \ \ \ \ \ \ }$t=\text{max}\{|x_1\cap  y_1|, |x_1\cap  y_2|\}$ when $0\leq i\leq D-1,\ j=D$,\  \\
\text{\ \ \ \ \ \ \ \ \ \ \ \ }$t=\text{max}\{|x_1\cap  y_1|=|x_2\cap  y_2|, |x_1\cap  y_2|=|x_2\cap  y_1|\}$ when $i=j=D$.
\begin{pro}\label{pro 16}
The sets {\rm $X^{\textbf{0}}_{i,j,t}$}, $(i,j,t)\in \mathcal{I}$ are the orbits
of $X\times X$ under the action
of {\rm Aut$_{\textbf{0}}$}$(X)$.
\end{pro}
\begin{proof}
Immediate from Proposition \ref{pro 15}.
\end{proof}
\begin{defi}\label{def1}
For each $(i,j,t)\in \mathcal{I}$, define the matrice $M^t_{i,j}\in {\rm{Mat}}_X(\mathbb{C})$ by
\begin{equation}\nonumber
(M^t_{i,j})_{xy}=\left\{\begin{array}{ll} 1 &\text{if}\ (x, y)\in X^\textbf{0}_{i,j,t},\\
 0 &\text{otherwise } \end{array}\right.
\ \ (x, y\in X).
\end{equation}
\end{defi}
Note that the transpose of $M^t_{i,j}$ is $M^t_{j,i}$.
Let $\mathcal{A}$ be the linear space spanned by  the matrices $M^t_{ij}$, $(i,j,t)\in \mathcal{I}$.
It is easy to check that $\mathcal{A}$ is closed under addition, scalar, taking the adjoint and
 matrix multiplication which is implied by Proposition \ref{pro 16}. Therefore $\mathcal{A}$ is a matrix $\mathbb{C}$$\ast$-algebra with
 the basis $M^t_{i,j}$. Next, we show that $\mathcal{A}$ coincides with $T$, where $T:=T$$(${\rm \textbf{0}}$)$ is the Terwilliger algebra of $\square_{2D}$. To do this, we need the following propositions.
Let $A_1$  and  $E^*_i=E^*_i$$(${\rm \textbf{0}}$)$ $(0\leq i\leq D)$ denote the adjacency matrix and the  $i${\rm th} dual idempotent, respectively.
\begin{pro}\label{pro2} With Definition {\rm \ref{def1}}, we have
\begin{itemize}
\item[\rm (i)] $M^{i}_{i,i}=E^*_i\ (0\leq i\leq D)$;
\item[\rm (ii)]
$M^{i-1}_{i-1, i}=E^*_{i-1}A_1E^*_{i}$,\ $M^{i-1}_{i, i-1}=E^*_{i}A_1E^*_{i-1}\ (0\leq i\leq D)$.
\end{itemize}
\end{pro}
\begin{proof}
(i) It follows from that the $(x,y)$-entry of $M^{i}_{i,i}$ is $1$ if $x=y$, $|x_1|=i$ and $0$ otherwise.

(ii) Consider the $(x,y)$-entry of both $M^{i-1}_{i-1, i}$ and $E^*_{i-1}A_1E^*_{i}$. For
$0\leq i\leq D-1$, we have $(M^{i-1}_{i-1, i})_{xy}=(E^*_{i-1}A_1E^*_{i})_{xy}$ is $1$ if $|x_1|=i-1,|y_1|=i$, $|x_1\cap y_1|=i-1$ and $0$ otherwise. For $i=D$, we have $(M^{D-1}_{D-1, D})_{xy}=(E^*_{D-1}A_1E^*_{D})_{xy}$ is $1$ if $|x_1|=D-1,|y_1|=|y_2|=D$, max$\{|x_1\cap y_1|,|x_1\cap y_2|\}=D-1$ and $0$ otherwise.
\end{proof}
\begin{pro}\label{pro 5}
With  Definition {\rm \ref{def1}}, we have
\begin{itemize}
\item[\rm (i)] $M^k_{k+i,k}=\frac{1}{i!}M^{k+i-1}_{k+i,k+i-1}\cdots M^{k+1}_{k+2,k+1}M^{k}_{k+1,k}\ \ (k\neq 0,i\geq 1)\ {\rm or} \ (k=0, 1\leq i\leq D-1)$;
\item[\rm (ii)] $M^0_{D,0}=\frac{1}{2D!}M^{D-1}_{D,D-1}\cdots M^{1}_{2,1}M^{0}_{1,0}$;
\item[\rm (iii)] $M^{k-i}_{k-i,k}=\frac{1}{i!}M^{k-i}_{k-i,k-i+1}M^{k-i+1}_{k-i+1,k-i+2}\cdots M^{k-1}_{k-1,k}\ \ (1\leq i<k\leq D) \ {\rm or} \ (1\leq k=i\leq D-1)$.
\end{itemize}
\end{pro}
\begin{proof}
(i) It is easy to verify $M^{k+1}_{k+2,k+1}M^{k}_{k+1,k}=2M^k_{k+2,k}$ since the entry of this matrix in  position $(x,y)$, with $|x_1|=k+2$
and $|y_1|=k$, is equal to $|\{z\in X||z_1|=k+1,y_1\subseteq z_1\subseteq x_1\}|$ if $k+2<D$ or
$|\{z\in X||z_1|=k+1,y_1\subseteq z_1\subseteq x_1\ \text{or}\ y_1\subseteq z_1\subseteq x_2\}|$  if $k+2=D$. Then by induction on $i\ ((k\neq 0,i\geq 1)$ or $(k=0, 1\leq i\leq D-1))$ we can obtain the desired result.\\
(ii) By use of (i), we first have $M^{D-2}_{D-1,D-2}\cdots M^{0}_{1,0}=(D-1)!M^{0}_{D-1,0}$.           Then we have $M^{D-1}_{D,D-1}M^{0}_{D-1,0}=2DM^{0}_{D,0}$ since the entry of this matrix in  position $(x,y)$, with $|x_1|=|x_2|=D$
and $|y_1|=0$, is equal to $|\{z\in X||z_1|=D-1,z_1\subseteq x_1 \ \text{or}\ z_1\subseteq x_2)\}|=2D$.\\
(iii) By taking transpose of both sides of (i) and replacing $k$ by $k-i$, we can obtain the desired result.
\end{proof}
\begin{pro}\label{pro 8}
With  Definition {\rm \ref{def1}}, we have
\begin{itemize}
\item[\rm (i)] for $0\leq i,j\leq D-1$,
\begin{align}\nonumber
M_{i,j}^{t}=\sum_{k=0}^{D-1}(-1)^{k-t}{k\choose t}M_{i,k}^{k}M_{k,j}^{k};
\end{align}
\item[\rm (ii)] for $i=D,0\leq j\leq D-1$ and $t\geq \lfloor\frac{j}{2}\rfloor+1$,
\begin{align}\nonumber
M_{D,j}^{t}=\sum_{k=\lfloor\frac{j}{2}\rfloor+1}^{D-1}(-1)^{k-t}{k\choose t}M_{D,k}^{k}M_{k,j}^{k};
\end{align}
\item[\rm (iii)] for $i=D,0\leq j\leq D-1$ and $t=\frac{j}{2}$ {\rm($j$ even)},
\begin{align}\nonumber
M_{D,j}^{\frac{j}{2}}=\frac{1}{2}\sum_{k=\frac{j}{2}}^{D-1}(-1)^{k-\frac{j}{2}}{k\choose \frac{j}{2}}M_{D,k}^{k}M_{k,j}^{k};
\end{align}
\item[\rm (iv)] for $0\leq i\leq D-1,j=D$ and $t\geq \lfloor\frac{i}{2}\rfloor+1$,
\begin{align}\nonumber
M_{i,D}^{t}=\sum_{k=\lfloor\frac{i}{2}\rfloor+1}^{D-1}(-1)^{k-t}{k\choose t}M_{i,k}^{k}M_{k,D}^{k};
\end{align}
\item[\rm (v)] for $0\leq i\leq D-1,j=D$ and $t=\frac{i}{2}$ {\rm($i$ even)},
\begin{align}\nonumber
M_{i,D}^{\frac{i}{2}}=\frac{1}{2}\sum_{k=\frac{i}{2}}^{D-1}(-1)^{k-\frac{i}{2}}{k\choose \frac{i}{2}}M_{i,k}^{k}M_{k,D}^{k};
\end{align}
\item[\rm (vi)] for $i=j=D$ and $t\geq \lfloor\frac{D}{2}\rfloor+1$,
\begin{align}\nonumber
M_{D,D}^{t}=\frac{1}{2}\bigg(\sum_{k=\lfloor\frac{D}{2}\rfloor+1}^{D}(-1)^{k-t}{k\choose t}M_{D,k}^{k}M_{k,D}^{k}+(-1)^{D-t}{D\choose t}M_{D,D}^{D}\bigg);
\end{align}
\item[\rm (vii)] for $i=j=D$ and $t=\frac{D}{2}$ {\rm($D$ even)},
\begin{align}\nonumber
M_{D,D}^{\frac{D}{2}}=\frac{1}{4}\bigg(\sum_{k=\frac{D}{2}}^{D}(-1)^{k-\frac{D}{2}}{k\choose \frac{D}{2}}M_{D,k}^{k}M_{k,D}^{k}+(-1)^{\frac{D}{2}}{D\choose \frac{D}{2}}M_{D,D}^{D}\bigg).
\end{align}
\end{itemize}
\end{pro}
\begin{proof}
(i) For $0\leq i,j\leq D-1$, we  have
$M_{i,k}^{k}M_{k,j}^{k}=\sum_{l=0}^{D-1}{l\choose k}M^l_{ij}$
since the entry of this matrix in  position $(x,y)$, with $|x_1|=i$
and $|y_1|=j$, is equal to $|\{z\in X||z_1|=k,z_1\subseteq (x_1\cap y_1)\}|$. It follows from Lemma \ref{lem 2.8}(ii) that
\begin{align}
\sum_{k=0}^{D-1}(-1)^{k-t}{k\choose t}M_{i,k}^{k}M_{k,j}^{k}
           &=\sum_{k=0}^{D-1}(-1)^{k-t}{k\choose t}\sum_{l=0}^{D-1}{l\choose k}M^l_{ij}\nonumber\\
           &=\sum_{l=0}^{D-1}\delta_{l,t}M^l_{ij}   \nonumber\\
           &=M_{i,j}^{t}.\nonumber
\end{align}

For  cases (ii)--(vii),  the proofs are similar to that of (i). Note that for $0\leq j\leq D-1$ $M_{D,k}^{k}M_{k,j}^{k}=\sum^{D-1}_{l=0}\big({l\choose k}+{j-l\choose k}\big)M_{D,j}^{l}$
($l\geq \lfloor\frac{j+1}{2}\rfloor$)
since the entry of this matrix in  position $(x,y)$, with $|x_1|=|x_2|=D$
and $|y_1|=j$, is equal to $|\{z\in X||z_1|=k,\ z_1\subseteq (x_1\cap y_1)\ \text{or}\ z_1\subseteq (x_2\cap y_1)\}|$; for  $1\leq k\leq D$,
$M_{D,k}^{k}M_{k,D}^{k}=\sum^D_{l=0}2\big({l\choose k}+{D-l\choose k}\big)M_{D,D}^{l}-{k\choose D}M_{D,D}^{D}$
($l\geq \lfloor\frac{D+1}{2}\rfloor$) since the entry of this matrix in  position $(x,y)$, with $|x_1|=|x_2|=D$
and $|y_1|=|y_2|=D$, is equal to $|\{z\in X||z_1|=k,z_1\subseteq (x_1\cap y_1)\ \text{or}\ z_1\subseteq (x_2\cap y_1) \ \text{or}\ z_1\subseteq (x_1\cap y_2)\ \text{or}\ z_1\subseteq (x_2\cap y_2)\}|$.
\end{proof}
\begin{thm}\label{thm2}
For $\square_{2D}$, the algebras $\mathcal{A}$ and $T$ coincide.
\end{thm}
\begin{proof}
 On the one hand, we have  $T\subseteq \mathcal{A}$ since $A_1=\sum^{D}_{i=1}(M^{i-1}_{i,i-1}+M^{i-1}_{i-1,i})$ and $E^*_i=M^{i}_{i,i}\ (0\leq i\leq D)$
 by Proposition \ref{pro2}.
 On the other hand, by Propositions \ref{pro2}--\ref{pro 8} we have  $\mathcal{A}\subseteq T$ since each $M^t_{ij}\in T$ for $(i,j,t)\in \mathcal{I}$.
 So the algebras $\mathcal{A}$ and $T$ coincide.
\end{proof}
\subsection{The Terwilliger algebra of $\square_{2D+1}$}
Recall the definition of $X$ for $n=2D+1$ and we view $X$ as the set consisting of all ordered pairs $(u,u')$ with $|u|<|u'|$.
 To each ordered triple $(x,y,z)\in X\times X\times X$, where
 $x:=(x_1,x_2), y:=(y_1,y_2), z:=(z_1,z_2)$,
 define $\partial(x,y,z)=(i,j,t)$: $i=\partial(x,y)$, $j=\partial(x,z)$, without loss of generality, let $|x_1\bigtriangleup y_1|=i$ and $|x_1\bigtriangleup z_1|=j$. Then  $t=|(x_1\bigtriangleup y_1) \cap  (x_1\bigtriangleup z_1)|$.
Observe that $0\leq t\leq i,j\leq D$
 and $\partial(y,z)=\text{min}\{i+j-2t,\  2D+1-(i+j-2t)\}$.
The set of three-tuples
$(i,j,t)$ that occur as $\partial(x,y,z)=(i,j,t)$ for some $x,y,z\in X$ is given by
$\mathcal{I}':=\{(i,j,t)|0\leq t\leq i,j\leq D, i+j-t\leq 2D \}.$
\begin{pro}\label{pro 1}
We have
$|\mathcal{I}'|=\frac{(D+1)(D+2)(2D+3)}{6}$.
\end{pro}
\begin{proof}
Similar to the proof of Proposition \ref{pro 3}(i), (ii):
$|\mathcal{I}'|=\sum^D_{l=0}{l+2\choose 2}+\sum^{2D}_{l=D+1}{2D-l+2\choose 2}\nonumber\\
=\frac{(D+1)(D+2)(2D+3)}{6}$.
\end{proof}
For each $(i,j,t)\in \mathcal{I}'$, define the sets
$X_{i,j,t}$ and
$X^\textbf{0}_{i,j,t}$  as in Subsection 3.1.
Note that $X^\textbf{0}_{i,j,t}=\{(x,y)\in X\times X| |x_1|=i,|y_1|=j, |x_1\cap y_1|=t\}$.
Similar to the proof of Proposition \ref{pro 15}, we have the following proposition.
\begin{pro}
The sets $X_{i,j,t}$, $(i,j,t)\in \mathcal{I}'$ are the orbits
of $X\times X\times X$ under the action
of {\rm Aut}$(X)$, where {\rm Aut}$(X)$ is the  automorphism group of $\square_{2D+1}$.
The sets {\rm $X^{\textbf{0}}_{i,j,t}$}, $(i,j,t)\in \mathcal{I}'$ are the orbits
of $X\times X$ under the action
of {\rm Aut$_{\textbf{0}}$}$(X)$, where {\rm Aut$_{\textbf{0}}$}$(X)$ is the stabilizer of vertex {\rm ${\textbf{0}}$} in ${\rm Aut}(X)$.
\end{pro}
\begin{defi}\label{def2}
For each $(i,j,t)\in \mathcal{I}'$, define the matrice $M^t_{i,j}\in {\rm{Mat}}_X(\mathbb{C})$ by
\begin{equation}\nonumber
(M^t_{i,j})_{xy}=\left\{\begin{array}{ll} 1 &\text{if}\ (x, y)\in X^\textbf{0}_{i,j,t},\\
 0 &\text{otherwise } \end{array}\right.
\ \ (x, y\in X).
\end{equation}
\end{defi}
Let $\mathcal{A}'$ be the linear space spanned by  the matrices $M^t_{ij}$, $(i,j,t)\in \mathcal{I}'$.
It is easy to check  that $\mathcal{A}'$  is a matrix $\mathbb{C}$$\ast$-algebra with
 the basis $M^t_{i,j}$, $(i,j,t)\in \mathcal{I}'$. We next show $\mathcal{A}'$ coincides with $T$, where $T:=T({\rm \textbf{0}})$ is the Terwilliger algebra of $\square_{2D+1}$. Let $A_1$ and $E^*_i=E^*_i$$(${\rm \textbf{0}}$)$ be the adjacency matrix and the $i${\rm th} dual idempotent of $\square_{2D+1}$, respectively.
\begin{pro}
With  Definition {\rm \ref{def2}}, we have
\begin{itemize}
\item[\rm (i)] $M^{i}_{i,i}=E^*_i\ (0\leq i\leq D)$;
\item[\rm (ii)]
$M^{i-1}_{i-1, i}=E^*_{i-1}A_1E^*_{i}$,\ $M^{i-1}_{i, i-1}=E^*_{i}A_1E^*_{i-1}\ (0\leq i\leq D)$;
\item[\rm (iii)] $M^k_{k+i,k}=\frac{1}{i!}M^{k+i-1}_{k+i,k+i-1}M^{k+i-2}_{k+i-1,k+i-2}\cdots M^{k}_{k+1,k}\ (1\leq i\leq D-k)$;
\item[\rm (iv)] $M^{k-i}_{k-i,k}=\frac{1}{i!}M^{k-i}_{k-i,k-i+1}M^{k-i+1}_{k-i+1,k-i+2}\cdots M^{k-1}_{k-1,k}\ (1\leq i\leq k)$;
\item[\rm (v)] $M_{i,j}^{t}=\sum_{k=0}^{D}(-1)^{k-t}{k\choose t}M_{i,k}^{k}M_{k,j}^{k}$.
\end{itemize}
\end{pro}
\begin{proof}
Similar to the proofs of Propositions \ref{pro2}, \ref{pro 5} and \ref{pro 8}(i).
\end{proof}
\begin{thm}
For $\square_{2D+1}$, the algebras $\mathcal{A}'$ and $T$ coincide.
\end{thm}
\begin{proof}
Similar to the proof of Theorem \ref{thm2}. Note that $A_1=\sum^{D}_{i=1}(M^{i-1}_{i,i-1}+M^{i-1}_{i-1,i})+M^{0}_{D,D}$.
\end{proof}
\section{Block diagonalization of $T$ of $\square_n$}
In this section,
we study a block-diagonalization of $T$ of $\square_{n}$ by using the theory of irreducible $T$-modules
together with the obtained basis in Section $3$. We treat two cases of $n$ even and odd separately.
\subsection{Block diagonalization of $T$ of $\square_{2D}$}
\begin{pro}\label{pro16}
For $\square_{2D}$, let $W$ denote  an irreducible $T$-module
with endpoint $r$ and diameter $d^*\ (0\leq r,d^*\leq D)$. Then $W$ is thin,
$r+d^*=D\ (r \text{even})$ or $r+d^*=D-1\ (r \text{odd})$, and the isomorphism class of $W$ is determined only by $r$.
\end{pro}
\begin{proof}
See \cite[Lemma 9.2, Theorem 13.1]{cau} and \cite[pp. 204--205]{ter3}. Note that the endpoint here is denoted by dual endpoint in \cite{ter3}.
\end{proof}
Based on Definition \ref{def1} and Proposition \ref{pro16}, for  $r = 0,1,\ldots,D$ define the linear vector space $\mathcal{L}_r$ as follows.
\begin{align}
\mathcal{L}_r:=\{\xi\in V:=\mathbb{C}^{X}|M^{r-1}_{r-1,r}\xi=0,\ \xi_{(x_1,x_2)}=0 \ \text{if}\ |x_1|\neq r\}.\nonumber
\end{align}
The  space $\mathcal{L}_r$ is in fact connected to the irreducible  $T$-modules. For discussional convenience, denote by $\mathcal{W}_r\ (0\leq r \leq D)$ the  $T$-module spanned by  all the irreducible $T$-modules with endpoint $r$, and define  $\mathcal{W}_r:=0$  if there does not exists such irreducible $T$-module.
\begin{pro}\label{pro3}
For $\square_{2D}$, let $W$  denote  an irreducible $T$-module
with endpoint $r$, diameter $d^*\ (0\leq r,d^*\leq D)$ and let $\mathcal{W}_r$ be defined as above. Then  the following {\rm (i)--(iv)} hold.
\begin{itemize}
\item[\rm(i)] $\mathcal{L}_r=E^*_r\mathcal{W}_r$.
\item[\rm(ii)] Up to isomorphism, $\mathcal{W}_r$ is ${2D\choose r}-{2D\choose r-1}$ copies of $W$ for $0\leq r\leq D-1$;  $\mathcal{W}_D$ is $\frac{1}{2}{2D\choose D}-\frac{D-1}{2D}{2D\choose D-1}$ copies of $W$ for $r=D$ $(D \ even)$;  $\mathcal{W}_D=0$ for $r=D$ $(D\ odd)$.
\item[\rm(iii)] Pick any $0\neq \xi\in \mathcal{L}_r$, then $0\neq M^r_{r+i,r}\xi\in E^*_{r+i}\mathcal{W}_r$ for $0\leq i\leq d^*$.
\item[\rm(iv)] Pick any $0\neq \xi\in \mathcal{L}_r$, then $M^{r-i}_{r-i,r}\xi=0$ for $1\leq i\leq r$.
\end{itemize}
\end{pro}
\begin{proof}
(i) We suppose $\mathcal{L}_r\neq 0$ and $\mathcal{W}_r\neq 0$. It is easy to see  that $0\neq \xi\in \mathcal{L}_r$ if and only if  $E^*_r\xi\neq 0$,
   $E^*_i\xi=0$ ($i\neq r$) and $E^*_{r-1}A_1E^*_r\xi=0$.
Pick any $0\neq \xi'\in E^*_r\mathcal{W}_r$. We have $\xi'\in \mathcal{L}_r$ since
$E^*_r\xi'\neq 0$,
   $E^*_i\xi'=0$ ($i\neq r$) and $E^*_{r-1}A_1E^*_r\xi'\in E^*_{r-1}(E^*_{r-1}\mathcal{W}_r+E^*_r\mathcal{W}_r+E^*_{r+1}\mathcal{W}_r)=0$, which is from Lemma \ref{lem1}(i),(ii). Thus $E^*_r\mathcal{W}_r\subseteq \mathcal{L}_r$. Conversely, pick any  $0\neq\xi'\in \mathcal{L}_r$.
 By  $E^*_r\xi'\neq 0$ and $E^*_i\xi'=0$ ($i\neq r$), we have $\xi'\in  E^*_r V$. Then
by $E^*_{r-1}A_1E^*_r\xi'=0$ and Lemma \ref{lem1}(ii),(iii), we have $\xi'\in  E^*_r\mathcal{W}_r$ since $V$ is the orthogonal direct sum of $\mathcal{W}_0+\mathcal{W}_1+\cdots+\mathcal{W}_{D}$.  Thus $\mathcal{L}_r\subseteq E^*_r\mathcal{W}_r$.\\
(ii) To prove this claim, it suffices to give the multiplicity of $W$ since the isomorphism class of $W$ is determined only by $r$. It is  clear that there exists a decomposition of irreducible $T$-modules for the standard module $V$:
\begin{equation}\label{hlheq30}
V=\sum^n_{h=0} W_{h} \ \ (\text {orthogonal\ direct\ sum}),
\end{equation}
Applying $E^*_r\ (0\leq r\leq D)$ to the both sides of \eqref{hlheq30}, we obtain dim$(E^*_rV)=\sum^n_{h=0} {\rm dim}(E^*_rW_{h})$.\\
(iia) For $0\leq r\leq D-1$, by Proposition \ref{pro16} we know that for each $h\ (0\leq h\leq n)$, dim$(E^*_rW_{h})=1$
if the endpoint of $W_h$ is at most $r$, and dim$(E^*_rW_{h})=0$
if the endpoint of $W_h$ is greater than $r$. Moreover, for every $\rho\ (0\leq \rho\leq D)$, there
exist exactly $m(\rho,d_{\rho})$ modules in \eqref{hlheq30} with endpoint $\rho$ and diameter $d_{\rho}$, where $m(\rho,d_{\rho})$ denotes the multiplicity of the module with endpoint $\rho$ and diameter $d_{\rho}$. Thus we have
\begin{equation}\label{hlheq33}
{\rm dim}(E^*_rV)=\sum_{\rho\leq r} m(\rho,d_{\rho}),
\end{equation}
which implies
\begin{align}
\ \ \ \ \ \ \ \ \ \ \ \ \ \ \ \ \ \ \ \  \ \ \ \ \ \ \  m(r,d^*)&={\rm dim}(E^*_rV)-{\rm dim}(E^*_{r-1}V)\nonumber\\
      &={2D\choose r}-{2D\choose r-1}.\nonumber\ \ (\text{by \cite[p, 264]{bcn} and \cite[p, 195]{ban}})
\end{align}
(iib) For $r=D$, it is easy to see $m(D,0)=0$ if $D$ is odd. Now, we suppose that $D$ is even. Similar to obtaining  \eqref{hlheq33}, we have ${\rm dim}(E^*_DV)=\sum_{\stackrel{\rho\leq D} {\rho\ even}}m(\rho,D-\rho)$.
So
\begin{align}
 m(D,0)&={\rm dim}(E^*_DV)-(m(0,D)+m(2,D-2)+\cdots+m(D-2,2))\nonumber\\
       &=\frac{1}{2}{2D\choose D}-\frac{D-1}{2D}{2D\choose D-1}.\nonumber
\end{align}
(iii)  Immediate from above (i), Proposition \ref{pro 5}(i),(ii) and Lemma \ref{lem2}.\\
(iv)  Immediate from above (i), Proposition \ref{pro 5}(iii) and Lemma \ref{lem1}(ii).
\end{proof}
\begin{cor}\label{pro 2}
For $\square_{2D}$, the following {\rm (i), (ii)} hold.
\begin{itemize}
\item[\rm (i)] For $0\leq r\leq D-1$,  $dim(\mathcal{L}_r)={2D\choose r}-{2D\choose r-1}.$
\item[\rm (ii)] For $r=D$,  $dim(\mathcal{L}_D)=\left\{\begin{array}{ll}\frac{1}{2}{2D\choose D}-\frac{D-1}{2D}{2D\choose D-1} &\text{if $D$ is even}\\[.2cm]
 0 &\text{if $D$ is odd}.\end{array}\right.$
\end{itemize}
\end{cor}
\begin{proof}
Immediate from Proposition \ref{pro3}(i), (ii).
\end{proof}

Propositions \ref{pro16}, \ref{pro3} and Corollary \ref{pro 2} imply the block sizes and block multiplicity of $T$. To describe this block diagonalization. We need consider the action of   matrices $M^t_{ij}$,  $(i,j,t)\in \mathcal{I}$
 on $M^r_{j,r}\xi$, where $0\neq \xi\in \mathcal{L}_r\ (0\leq r\leq D)$.
\begin{pro}\label{pro 9}
For all $(i,j,t)\in \mathcal{I}$, $r\in\{0,1,\ldots,D\}$ and for $\xi\in \mathcal{L}_r$, we have
\begin{itemize}
\item[\rm (i)] for $0\leq i,j\leq D-1$,
\begin{align}
 {2D-2r\choose i-r}M_{i,j}^{t}M_{j,r}^{r}\xi=\beta^r_{i,j,t}M_{i,r}^{r}\xi,\nonumber
\end{align}
where $\beta_{i,j,t}^{r}={2D-2r\choose i-r}\sum^{D-1}_{l=0}(-1)^{r-l}{r\choose l}{i-l\choose t-l}{2D-i-r+l\choose j-r-t+l};$
\item[\rm (ii)] for $i=D, 0\leq j\leq D-1$,
\begin{align*}
2{2D-2r\choose D-r} M_{D,j}^{t}M_{j,r}^{r}\xi=\beta^r_{D,j,t}M_{D,r}^{r}\xi,
\end{align*}
where $\beta_{D,j,t}^{r}=2{2D-2r\choose D-r}\big(\sum^{D-1}_{l=\lfloor\frac{r}{2}\rfloor+1}(-1)^{r-l}{r\choose l}\big({D-l\choose t-l}{D-r+l\choose j-t-r+l}+{D-r+l\choose t-r+l}{D-l\choose j-t-l}\big)\\
 \text{\ \ \ \ \ \ \ \  \ \ \ \ \ \ \ \ \ \ \ \ \ }    +{D-\frac{r}{2}\choose t-\frac{r}{2}}{D-\frac{r}{2}\choose j-t-\frac{r}{2}}(-1)^{\frac{r}{2}}{r\choose \frac{r}{2}}\big)$;
 \item[\rm (iii)] for $0\leq i\leq D-1, j=D$,
\begin{align*}
{2D-2r\choose i-r} M_{i,D}^{t}M_{D,r}^{r}\xi=\beta^r_{i,D,t}M_{i,r}^{r}\xi,
\end{align*}
where $\beta^r_{i,D,t}={2D-2r\choose i-r}\big(\sum^{D-1}_{l=0}(-1)^{r-l}{r\choose l}\big({i-l\choose t-l}{2D-r-i+l\choose D-r-t+l}+{i-l\choose t}{2D-r-i+l\choose D-t}\big)\big)$;
\item[\rm (iv)] for $i=j=D$ and $0\leq r\leq D-1$,
\begin{align*}
2{2D-2r\choose D-r} M_{D,D}^{t}M_{D,r}^{r}\xi=\beta^r_{D,D,t}M_{D,r}^{r}\xi,
\end{align*}
where $\beta^r_{D,D,t}=2{2D-2r\choose D-r}\big(\sum^{D-1}_{l=\lfloor\frac{r}{2}\rfloor+1}2(-1)^{r-l}{r\choose l}\big({D-l\choose t-l}{D-r+l\choose t}+{D-l\choose t}{D-r+l\choose D-t}\big)\\ \nonumber
  \text{\ \ \ \ \ \ \ \  \ \ \ \ \ \ \ \ \ \ \ \ \ }         +2(-1)^{\frac{r}{2}}{r\choose \frac{r}{2}}{D-\frac{r}{2}\choose D-t}{D-\frac{r}{2}\choose t}\big)$.
\item[\rm (v)] for $i=j=D$ and $r=D\ (D\ is\ even)$,
  \begin{align*}
 M_{D,D}^{t}M_{D,D}^{D}\xi=\beta^D_{D,D,t}\xi,
\end{align*}
where $\beta^D_{D,D,t}=(-1)^{D-t}{D\choose t}$ if  $t\geq \frac{D}{2}+1$ and  $\beta^D_{D,D,t}=\frac{1}{2}(-1)^{\frac{D}{2}}{D\choose \frac{D}{2}}$ if $t=\frac{D}{2}$.
\end{itemize}
\end{pro}
\begin{proof}
(i) For $0\leq i, j\leq D-1$, we first have $M_{i,j}^{t}M_{j,r}^{r}\xi=\sum^{D-1}_{l=0}{i-l\choose t-l}{2D-i-r+l\choose j-t-r+l}M_{i,r}^{l}\xi.$
Then by Propositions \ref{pro 8}(i) and \ref{pro3}(iv), we have $M_{i,r}^{l}\xi=(-1)^{r-l}{r\choose l}M_{i,r}^{r}\xi$. So
\begin{align*}
 M_{i,j}^{t}M_{j,r}^{r}\xi=\sum^{D-1}_{l=0}{i-l\choose i-t}{2D-i-r+l\choose2D-i-j+t}(-1)^{r-l}{r\choose l}M_{i,r}^{r}\xi.
\end{align*}
For cases (ii)--(iv), by the argument similar to proof of case (i) we can obtain the desired results.
(v) is immediate from Proposition \ref{pro 8}(vi), (vii). Note that $M_{D,D}^{D}\xi=\xi$.
\end{proof}
In the following, we describe a block-diagonalization of $T$ of $\square_{2D}$. We first consider the case $D$  even.
\subsubsection{Block diagonalization of $T$ of $\square_{2D}$ with even $D$}
In this subsection, we suppose $D> 3$ is even. Based on Propositions \ref{pro16}, \ref{pro3} and Corollary \ref{pro 2}, for each $r = 0,1,\ldots,D$
denote by  $B_r$ the set of an orthonormal basis  of $\mathcal{L}_r$ and let
\begin{align}\nonumber
\mathcal{B}_1=\{(r,\xi,i)|r = 0,1,\ldots,D,\ \xi\in B_r,\ &i=r,r+1,\ldots,D\ \text{for even $r$}\\
 &i=r,r+1,\ldots,D-1\ \text{for odd $r$} \}.\nonumber
\end{align}
It is not difficult to calculate
\begin{align}
&|\mathcal{B}_1|=\sum^{D-2}_{\stackrel{r=0} {r\ even}}(D-r+1)\bigg({2D\choose r}-{2D\choose r-1}\bigg)+\bigg(\frac{1}{2}{2D\choose D}-\frac{D-1}{2D}{2D\choose D-1}\bigg)\nonumber\\
&\ \ \ \ \ \ \ \ \    +\sum^{D-1}_{\stackrel{r=1} {r\ odd}}(D-r)\bigg({2D\choose r}-{2D\choose r-1}\bigg)
             =2^{2D-1}.\label{eq11}
\end{align}
For each $(r,\xi,i)\in\mathcal{B}_1$, define the vector $u_{r,\xi,i}\in V$ by
\begin{align}
u_{r,\xi,i}&:={2D-2r\choose i-r}^{-\frac{1}{2}}M^r_{i,r}\xi \ \ (r\leq i\leq D-1),\label{eq16}\\
u_{r,\xi,D}&:=\frac{\sqrt{2}}{2}{2D-2r\choose D-r}^{-\frac{1}{2}}M^r_{D,r}\xi \ \ (i=D\ \text{and}\ 0\leq r<D\ \text{ even}),\label{eq19}\\
u_{D,\xi,D}&:=\xi\ \ (i=r=D).
\end{align}
\begin{pro}\label{pro 14}
The vectors $u_{r,\xi,i}, \ (r,\xi,i)\in \mathcal{B}_1$ form an orthonormal
basis of the standard module $V$.
\end{pro}
\begin{proof}
For $r\leq i\leq D-1$,
\begin{align}
\xi^{\rm T}M^r_{r,i}M^r_{i,r}\xi&=\sum^r_{l=0}{2D-2r+l\choose i-2r+l}\xi^{\rm T}M^l_{r,r}\xi\ \ \ \ \ \nonumber\\
&=\sum^r_{l=0}{2D-2r+l\choose 2D-i}(-1)^{r-l}{r\choose l}\xi^{\rm T}\xi \nonumber\ \ (\text{by\ Propositions \ref{pro 8}(i) and \ref{pro3}(iv)})\\
&={2D-2r\choose i-r}\xi^{\rm T}\xi;\ \ \ \ \ \ \ \ \  \text{(by Lemma \ref{lem 2.8}(iii))}\nonumber\\
\ \ \ \ \ \ \ \ \ \ \   \text{For $i=D$},\nonumber\\
\xi^{\rm T}M^r_{r,D}M^r_{D,r}\xi&=\sum^r_{l=0}{2D-2r+l\choose D-2r+l}\xi^{\rm T}M^l_{r,r}\xi+{2D-2r\choose D-r}\xi^{\rm T}M^0_{r,r}\xi\ \ \ \ \ \nonumber\\
&=2{2D-2r\choose D-r}\xi^{\rm T}\xi.\ \ \ \ \ \ \ \   \text{(by Propositions \ref{pro 8}(i), \ref{pro3}(iv), Lemma \ref{lem 2.8}(iii))} \nonumber
\end{align}
It follows that $u_{r,\xi,i}, \ (r,\xi,i)\in \mathcal{B}_1$  are normal.
Next, we show that $u_{r,\xi,i}$  are pairwise orthogonal. By Proposition \ref{pro3}(i),(iii), the vectors $u_{r,\xi,i}$ and $u_{r',\xi',i'}$  are orthogonal if $r\neq r'$ or $i\neq i'$. One can easily verifies that $u_{r,\xi,i}$ and $u_{r',\xi',i'}$  are also orthogonal if $r=r',i=i'$, $\xi \neq \xi'$ by the argument similar to the
proof of normality since  $\xi^{\rm T}\xi'=0$.
\end{proof}
Let $U_1$ be the $X\times \mathcal{B}_1$ matrix  with $u_{r,\xi,i}$ as the $(r,\xi,i)$-th column.
For each triple $(i,j,t)\in \mathcal{I}$, define the matrix
$\widetilde{M}_{i,j}^{t}:=U_1^{\rm T}M_{i,j}^{t}U_1$. The following proposition shows that $\widetilde{M}_{i,j}^{t}$ is in block diagonal form.
\begin{pro}\label{pro 6}
For   $(i,j,t)\in \mathcal{I}$ and $(r,\xi,i'),(r',\xi',j')\in \mathcal{B}_1$, the following {\rm(i)--(iv)} hold.
\begin{itemize}
\item[\rm (i)] For $0\leq i,j\leq D-1$,
\begin{align*}
(\widetilde{M}_{i,j}^{t})_{(r,\xi,i'),(r',\xi',j')}=\left\{\begin{array}{ll}
{2D-2r\choose i-r}^{-\frac{1}{2}}
                     {2D-2r\choose j-r}^{-\frac{1}{2}}
                     \beta_{i,j,t}^{r}
\  &if\  r=r',\xi=\xi', i=i',j=j',\\
0\   &otherwise.
\end{array}
\right.
\end{align*}
\item[\rm (ii)] For $i=D, 0\leq j\leq D-1$,
\begin{align*}
(\widetilde{M}_{D,j}^{t})_{(r,\xi,i'),(r',\xi',j')}=\left\{\begin{array}{ll}
\frac{\sqrt{2}}{2}{2D-2r\choose D-r}^{-\frac{1}{2}}
                     {2D-2r\choose j-r}^{-\frac{1}{2}}
                     \beta_{D,j,t}^{r}
\  &if\  r=r',\xi=\xi', i'=D,j=j',\\
0\   &otherwise.
\end{array}
\right.
\end{align*}
\item[\rm (iii)] For $0\leq i\leq D-1, j=D$,
\begin{align*}
(\widetilde{M}_{i,D}^{t})_{(r,\xi,i'),(r',\xi',j')}=\left\{\begin{array}{ll}
\frac{\sqrt{2}}{2}{2D-2r\choose D-r}^{-\frac{1}{2}}
                     {2D-2r\choose i-r}^{-\frac{1}{2}}
                     \beta_{i,D,t}^{r}
\  &if\  r=r',\xi=\xi', i=i',j'=D,\\
0\   &otherwise.
\end{array}
\right.
\end{align*}
\item[\rm (iv)] For $i=j=D$ and $0\leq r\leq D-1$,
\begin{align*}
(\widetilde{M}_{D,D}^{t})_{(r,\xi,i'),(r',\xi',j')}=\left\{\begin{array}{ll}
\frac{1}{2}{2D-2r\choose D-r}^{-1}\beta_{D,D,t}^{r}
\  &if\  r=r',\xi=\xi', i'=j'=D,\\
0\   &otherwise.
\end{array}
\right.
\end{align*}
\item[\rm (v)] For $i=j=D$ and $r=D$,
\begin{align*}
(\widetilde{M}_{D,D}^{t})_{(r,\xi,i'),(r',\xi',j')}=\left\{\begin{array}{ll}
\beta_{D,D,t}^{D}
\  &if\  r=r'=D,\xi=\xi', i'=j'=D,\\
0\   &otherwise.
\end{array}
\right.
\end{align*}
\end{itemize}
Note that the numbers  $\beta_{i,j,t}^{r}$ are from Proposition {\rm \ref{pro 9}} and $r$ is even in  {\rm (ii)--(v)}.
\end{pro}
\begin{proof}(i) For $0\leq i,j\leq D-1$, it is clear that
$(\widetilde{M}_{i,j}^{t})_{(r,\xi,i'),(r',\xi',j')}=u_{r,\xi,i'}^{\rm T}M_{i,j}^{t}u_{r',\xi',j'}$.
By \eqref{eq16}, we have
\begin{align}
M_{i,j}^{t}u_{r',\xi',j'}&={2D-2r'\choose j'-r'}^{-\frac{1}{2}}M_{i,j}^{t}M_{j',r'}^{r'}\xi' \nonumber\\
                     &=\delta_{j,j'}{2D-2r'\choose j-r'}^{-\frac{1}{2}}{2D-2r'\choose i-r'}^{-1}
                     \beta^{r'}_{i,j,t}M_{i,r'}^{r'}\xi'\ \ \ \ \ \ \ ({\rm \text by\ Proposition
                     \ \ref{pro 9}(i)})\nonumber\\
                     &=\delta_{j,j'}{2D-2r'\choose j-r'}^{-\frac{1}{2}}
                     {2D-2r'\choose i-r'}^{-\frac{1}{2}}
                     \beta_{i,j,t}^{r'}u_{r',\xi',i},\nonumber
\end{align}
from which (i) follows.

The proofs of (ii)--(v) are similar to that of (i).
\end{proof}
 Proposition \ref{pro 6} implies that each matrix $\widetilde{M}_{i,j}^{t}$, $(i,j,t)\in \mathcal{I}$ has a block diagonal form:
  for each  even $0\leq r\leq D-1$ there are ${2D\choose r}-{2D\choose r-1}$ copies of
 a $(D+1-r)\times(D+1-r)$ block on the diagonal; for each odd $0\leq r\leq D-1$ there are ${2D\choose r}-{2D\choose r-1}$ copies of
 a  $(D-r)\times(D-r)$  block on the diagonal; for  $r=D$ there are $\frac{1}{2}{2D\choose D}-\frac{D-1}{2D}{2D\choose D-1}$ copies of
 a $1\times 1$  block on the diagonal. For each $r$ the copies are indexed by the elements of $B_r$, and in each copy
the rows and columns are indexed by the integers $i\in \{r,r+1,\ldots, D\}$ ($r$  even) or $i\in \{r,r+1,\ldots, D-1\}$ ($r$  odd). Thus
 by deleting copies of blocks and using
 the identity $\sum^{D}_{\stackrel{r=0} {r\ even}}(D-r+1)^2+\sum^{D-1}_{\stackrel{r=1} {r\ odd}}(D-r)^2=\frac{(D+1)(D^2+2D+3)}{3}$, we have the following theorem.
\begin{thm}\label{thm 2}
For  $\square_{2D}$ with even $D>3$, the above matrix $U_1$ gives a block-diagonalization of $T$ and $T$ is isomorphic to
$\bigoplus^D_{r=0}{\mathbb{C}}^{N_r\times N_r}$, where $N_r:=\{r,r+1,\ldots, D\}$ $(r\ even)$ or  $N_r:=\{r,r+1,\ldots, D-1\}$ $(r\ odd)$.
\end{thm}
\subsubsection{Block diagonalization of $T$ of $\square_{2D}$ with odd $D$}
In this subsection, we suppose $D\geq 3$ is odd.
Based on Propositions \ref{pro16}, \ref{pro3} and Corollary \ref{pro 2}, for each $r = 0,1,\ldots,D-1$,
denote by  $B_r$ the set of an orthonormal basis  of $\mathcal{L}_r$ and let
\begin{align}\nonumber
\mathcal{B}_2=\{(r,\xi,i)|r = 0,1,\ldots,D-1,\ \xi\in B_r,\ &i=r,r+1,\ldots,D\ \text{for even $r$}\\
 &i=r,r+1,\ldots,D-1\ \text{for odd $r$} \}.\nonumber
\end{align}
It is not difficult to calculate
\begin{align}
|\mathcal{B}_2|&=\sum^{D-1}_{\stackrel{r=0} {r\ even}}(D-r+1)\bigg({2D\choose r}-{2D\choose r-1}\bigg)+\sum^{D-2}_{\stackrel{r=1} {r\ odd}}(D-r)\bigg({2D\choose r}-{2D\choose r-1}\bigg)\nonumber\\
             &=2^{2D-1}.
\end{align}

For each $(r,\xi,i)\in\mathcal{B}_2$, define the vector $u_{r,\xi,i}\in V$  by the forms of \eqref{eq16} and \eqref{eq19}. One can easily verifies that the vectors $u_{r,\xi,i}, \ (r,\xi,i)\in \mathcal{B}$ form an orthonormal
basis of the standard module $V$. Let $U_2$ be the $X\times \mathcal{B}_2$ matrix  with $u_{r,\xi,i}$ as the $(r,\xi,i)$-th column.  It follows from
Proposition  \ref{pro 6}(i)--(iv) that for each triple $(i,j,t)\in \mathcal{I}$ the matrix
$\widetilde{M}_{i,j}^{t}:=U_2^{\rm T}M_{i,j}^{t}U_2$ is in block diagonal form:
  for each even $0\leq r\leq D-1$ there are ${2D\choose r}-{2D\choose r-1}$ copies of
 a $(D+1-r)\times(D+1-r)$  block on the diagonal; for each odd $0\leq r\leq D-1$ there are ${2D\choose r}-{2D\choose r-1}$ copies of a $(D-r)\times(D-r)$  block on the diagonal.
 By deleting copies of blocks and using
 the identity $\sum^{D-1}_{\stackrel{r=0} {r\ even}}(D-r+1)^2+\sum^{D-2}_{\stackrel{r=1} {r\ odd}}(D-r)^2=\frac{(D+1)(D^2+2D+3)}{3}$, we have the following theorem.
\begin{thm}\label{thm3}
For  $\square_{2D}$ with odd $D\geq 3$, the above matrix $U_2$ gives a block diagonalization of $T$ and $T$ is isomorphic to
$\bigoplus^{D-1}_{r=0}{\mathbb{C}}^{N_r\times N_r}$, where $N_r:=\{r,r+1,\ldots, D\}$ $(r\ even)$ or  $N_r:=\{r,r+1,\ldots, D-1\}$ $(r\ odd)$.
\end{thm}
\subsection{Block diagonalization of $T$ of $\square_{2D+1}$}
\begin{pro}\label{pro21}
For $\square_{2D+1}$ with $D\geq 2$, let $W$ denote  an irreducible $T$-module
with endpoint $r$ and diameter $d^*\ (0\leq r,d^*\leq D)$. Then $W$ is thin,
$r+d^*=D$  and the isomorphism class of $W$ is determined only by $r$.
\end{pro}
\begin{proof}
From \cite{ca} we known that $W$ is thin, $r+d^*=D$ and the isomorphism class of $W$ is determined by its dual endpoint and $d^*$. By \cite[pp. 305-306]{ban} and \cite[p. 196]{ter3} we have that  $\square_{2D+1}$ is isomorphic to $\frac{1}{2}H(2D+1,2){'''}$. Then it follows from \cite[p. 204]{ter3} that
both $W$'s dual endpoint and $d^*$ can be determined  by $r$.
\end{proof}
Based on Definition \ref{def2} and Proposition \ref{pro21}, for  $r = 0,1,\ldots,D$, define the linear vector space $\mathcal{L}'_r$ as follows.
\begin{align}
\mathcal{L}'_r:=\{\xi\in V|M^{r-1}_{r-1,r}\xi=0,\xi_{(x_1,x_2)}=0 \ \text{if}\ |x_1|\neq r\}.\nonumber
\end{align}
\begin{pro}\label{pro15}
For $\square_{2D+1}$ with $D\geq 2$, let $W$  denote  an irreducible $T$-module
with endpoint $r$, diameter $d^*\ (0\leq r,d^*\leq D)$ and let $\mathcal{W}_r$ be defined as in Subsection $4.1$. Then  the following {\rm (i)--(iv)} hold.
\begin{itemize}
\item[\rm(i)] $\mathcal{L}'_r=E^*_r\mathcal{W}_r$.
\item[\rm(ii)] Up to isomorphism, $\mathcal{W}_r$ is ${2D\choose r}-{2D\choose r-1}$ copies of $W$ for $0\leq r\leq D$.
\item[\rm(iii)] Pick any $0\neq \xi\in \mathcal{L}'_r$, then $0\neq M^r_{r+i,r}\xi\in E^*_{r+i}\mathcal{W}_r$ for $0\leq i\leq d^*$.
\item[\rm(iv)] Pick any $0\neq \xi\in \mathcal{L}'_r$, then $M^{r-i}_{r-i,r}\xi=0$ for $1\leq i\leq r$.
\end{itemize}
\end{pro}
\begin{proof}
Similar to the proof of Proposition \ref{pro3}.
\end{proof}
\begin{cor}\label{pro 4}
We have  $dim(\mathcal{L}'_r)={2D+1\choose r}-{2D+1\choose r-1}$ for $0\leq r\leq D$.
\end{cor}
\begin{pro}\label{pro5}
For all $(i,j,t)\in \mathcal{I}'$, $r\in\{0,1,\ldots,D\}$ and for $\xi\in \mathcal{L}'_r$, we have
\begin{align}\nonumber
{2D+1-2r\choose i-r}M_{i,j}^{t}M_{j,r}^{r}\xi=\beta^r_{i,j,t}M_{i,r}^{r}\xi,
\end{align}
where  $\beta_{i,j,t}^{r}={2D+1-2r\choose i-r}\sum^D_{l=0}(-1)^{r-l}{r\choose l}{i-l\choose t-l}{2D+1+l-i-r\choose j-t-r+l}.$
\end{pro}
\begin{proof}
Similar to the proof of Proposition \ref{pro 9}(i).
\end{proof}
Based on Propositions \ref{pro21}, \ref{pro15} and Corollary \ref{pro 4}, for each $r = 0,1,\ldots,D$,
denote by  $B'_r$ the set of an orthonormal basis  of $\mathcal{L}'_r$ and let
 $\mathcal{B}'=\{(r,\xi,i)|r = 0,1,\ldots,D,\ \xi\in B'_r,\ i=r,r+1,\ldots,D\}$.
Then it is not difficult to calculate
\begin{align}\label{eq33}
|\mathcal{B}'|&=\sum^D_{r=0}(D-r+1)\bigg({2D+1\choose r}-{2D+1\choose r-1}\bigg)\nonumber\\
             &=2^{2D}.
\end{align}
For each $(r,\xi,i)\in\mathcal{B}'$, define the vector $u_{r,\xi,i}\in \mathbb{C}^{X}$ by
\begin{align}\label{eq12}
u_{r,\xi,i}:={2D+1-2r\choose i-r}^{-\frac{1}{2}}M^r_{i,r}\xi.
\end{align}
The form of $u_{r,\xi,i}$  is from
$\xi^{\rm T}M^r_{r,i}M^r_{i,r}\xi={2D+1-2r\choose i-r}\xi^{\rm T}\xi.$

By the argument similar to proof of Proposition \ref{pro 14}, we can easily prove that   the vectors $u_{r,\xi,i}, \ (r,\xi,i)\in \mathcal{B}'$ form an orthonormal
base of the standmodule $V$.
 Let $U'$ be the $X\times \mathcal{B}'$ matrix  with $u_{r,\xi,i}$ as the $(r,\xi,i)$-th column.
For each triple $(i,j,t)\in \mathcal{I}'$ define the martices
$\widetilde{M}_{i,j}^{t}:={U'}^{\rm T}M_{i,j}^{t}U'$.
\begin{pro}\label{pro9}
For $(i,j,t)\in \mathcal{I}'$ and $(r,\xi,i'),(r',\xi',j')\in \mathcal{B}'$,
\begin{align*}
(\widetilde{M}_{i,j}^{t})_{(r,\xi,i'),(r',\xi',j')}=\left\{\begin{array}{ll}
{2D+1-2r\choose i-r}^{-\frac{1}{2}}{2D+1-2r\choose j-r}^{-\frac{1}{2}}
                     \beta_{i,j,t}^{r}
\ \ \ \ &if\  r=r',\xi=\xi', i=i',j=j',\\
0\ \ \ \ &otherwise,
\end{array}
\right.
\end{align*}
where the numbers $\beta_{i,j,t}^{r}$ are from Proposition {\rm \ref{pro5}}.
\end{pro}
\begin{proof}
Similar to the proof of Proposition \ref{pro 6}(i).
\end{proof}
Proposition \ref{pro9} implies that each matrix $\widetilde{M}_{i,j}^{t}$, $(i,j,t)\in \mathcal{I}'$ has a block diagonal form:
for each $0\leq r\leq D$ there are ${2D+1\choose r}-{2D+1\choose r-1}$ copies of
an $(D+1-r)\times(D+1-r)$ block on the diagonal.
By deleting copies of blocks and using
the identity $\sum^D_{r=0}(D-r+1)^2=\frac{(D+1)(D+2)(2D+3)}{6}$, we have the following theorem.
\begin{thm}\label{thm1}
For $\square_{2D+1}$ with $D\geq 3$, the above matrix $U'$ gives a block-diagonalization of $T$ and $T$ is isomorphic to
$\bigoplus^D_{r=0}{\mathbb{C}}^{N_r\times N_r}$, where $N_r=\{r,r+1,\ldots, D\}$.
\end{thm}
\section{Semidefinite programming bound on $A(\square_n,d)$}
In this section, we give an  upper bound on $A(\square_{n},d)$ by semidefinite programming involving the block-diagonalization of $T$. We treat two cases of $n$ even and odd separately.
\subsection{Semidefinite programming bound on $A(\square_{2D},d)$}
Given  code $C$,  for each $(i,j,t)\in \mathcal{I}$ define the numbers
$\lambda^t_{i,j}:=|(C\times C\times C)\cap X_{i,j,t}|$ and numbers  $x^t_{i,j}:=(|C|\gamma^t_{i,j})^{-1}\lambda^t_{i,j}$, where $\gamma^t_{i,j}$ denotes the number of nonzero entries of $M^t_{ij}$. Observe that
\begin{align}
|C|=\sum^D_{i=0}\gamma^0_{i,0}x^0_{i,0}.
\end{align}
Define the matrix $M_C\in {\rm{Mat}}_X(\mathbb{C})$  by
\begin{equation}\nonumber
(M_C)_{xy}=\left\{\begin{array}{ll} 1 &\text{if}\ x,y\in C,\\
 0 &\text{otherwise.} \  \end{array}\right.
\end{equation}
Observe  that $M_C=\chi_c\chi^{\rm T}_c$ is positive semidefinite, where $\chi_c$ is the characteristic column vector of $C$.
In the following, we define two important matrices by
\begin{align}
M':=\frac{1}{|C||{\rm Aut}_{\textbf{0}}(X)|}\sum_{\stackrel {\sigma \in{\rm Aut}(X)}
{\textbf{0}\in \sigma C}}M_{\sigma C},\ \
M'':=\frac{1}{(|X|-|C|)|{\rm Aut}_{\textbf{0}}(X)|}\sum_{\stackrel {\sigma \in{\rm Aut}(X)}
{\textbf{0}\notin \sigma C}}M_{\sigma C}.\nonumber
\end{align}
Observe that the matrices $M'$ and $M''$ are positive semidefinite and  invariant under any permutation of ${\rm Aut}_{\textbf{0}}(X)$
of rows and columns, and  hence  they are in $T$ by Proposition \ref{pro 16}.
\begin{pro}\label{pro6} With above notation, we have
\begin{itemize}
\item[\rm (i)]$M'=\sum_{(i,j,t)\in \mathcal{I}}x^t_{i,j}M^t_{i,j}$.
\item[\rm(ii)]$M''=\frac{|C|}{|X|-|C|}\sum_{(i,j,t)\in\mathcal{I}}(x^0_{\zeta,0}-x^t_{i,j})M^t_{i,j}$,        where $\zeta=${\rm min}$\{i+j-2t,2D-(i+j-2t)\}$.
\end{itemize}
\end{pro}
\begin{proof}
(i) Let $\Phi=\{\sigma \in{\rm Aut}(X)|\textbf{0}\in \sigma C\}$. Let $x,y,z\in C$ and let $(x,y,z)\in  X_{i,j,t}$. Then
there exists $\sigma'\in \Phi$ that map $x$ to $\textbf{0}$ and hence  $(\sigma' y,\sigma' z)\in X^{\textbf{0}}_{i,j,t}$. If  $\psi\in  {\rm Aut}_{\textbf{0}}(X)$ ranges over the ${\rm Aut}_{\textbf{0}}(X)$, then $(\psi\sigma' y, \psi\sigma' z)$ ranges over $X^{\textbf{0}}_{i,j,t}$.
Note that the set $\{\psi\sigma'| \psi\in {\rm Aut}_{\textbf{0}}(X)\}$ consists of all automorphisms in $\Phi$ that map $x$ to $\textbf{0}$. Hence by $M'\in T$ we have
\begin{align}
M'&=\frac{1}{|C||{\rm Aut}_{\textbf{0}}(X)|}\sum_{(i,j,t)\in \mathcal{I}}
\frac{\lambda^t_{i,j}|{\rm Aut}_{\textbf{0}}(X)|}{\gamma^t_{i,j}}M^t_{i,j}.\nonumber\\
&=\sum_{(i,j,t)\in \mathcal{I}}x^t_{i,j}M^t_{i,j}\nonumber
\end{align}
(ii) Let $M =|C|M'+(|X|-|C|)M''$, that is
 $M =\frac{1}{|{\rm Aut_{\textbf{0}}}|}\sum_{\sigma \in{\rm Aut}(X)}M_{\sigma C}$.
Note that the matrice $M$ is ${\rm Aut}(X)$-invariant and hence an
element of the Bose-Mesner algebra of $\square_{2D}$, and
we write $M=\sum^D_{k=0}\alpha_kA_k$. Then for any $x\in X$ with $\partial(x,\textbf{0})=k$, we have
$\alpha_k=(M)_{x,\textbf{0}}=(|C|M')_{x,\textbf{0}}=|C|x^0_{k,0}.$
So
\begin{align}
M''&=\frac{1}{|X|-|C|}(M-|C|M')\nonumber\\
 &=\frac{1}{|X|-|C|}(\sum^D_{k=0}|C|x^0_{k,0}A_k-|C|\sum_{(i,j,t)\in \mathcal{I}}x^t_{i,j}M^t_{i,j})\nonumber\\
 &=\frac{|C|}{|X|-|C|}\sum_{(i,j,t)\in \mathcal{I}}(x^0_{\zeta,0}-x^t_{i,j})M^t_{i,j},\nonumber
\end{align}
where $\zeta=${\rm min}$\{i+j-2t,2D-(i+j-2t)\}$.
\end{proof}
\begin{pro}\label{pro4} $x^t_{i,j},(i,j,t)\in \mathcal{I}$ satisfy the following linear constraints,
where {\rm (v)} holds if $C$ has minimum distance at least $d$:
\begin{align}
{\rm (i)}\  &x^0_{0,0}=1.\nonumber\\
{\rm (ii)}\ & 0\leq x^t_{i,j}\leq x^0_{i,0}.\nonumber \\
{\rm (iii)}\ &\text{For}\ 0\leq i,j \leq D,\ 0\leq i+j-2t\leq D,\ x^t_{i,j}=x^{t'}_{i',j'}\ \label{eq21}\\
 &\text{if}\ (i',j',i'+j'-2t')\  \text{is a permutation of}\ (i,j,i+j-2t). \ \ \ \ \ \ \ \ \ \ \ \ \ \ \ \ \ \ \ \ \ \ \ \ \ \ \ \  \ \ \ \ \ \ \ \ \ \ \nonumber  \\
 {\rm (iv)}\ &\text{For}\ 0\leq i,j\leq D,\ D+1\leq i+j-2t\leq 2D-2,\ x^t_{i,j}=x^{t'}_{i',j'}\ \  \ \nonumber\\
 &\text{if} \ (i',j',2D-(i'+j'-2t'))\ \text{is a permutation of}\ (i,j,2D-(i+j-2t)). \ \ \ \ \ \ \ \ \ \ \ \ \ \ \ \ \ \ \ \ \ \ \ \ \ \ \ \  \ \ \ \ \ \ \ \ \ \ \nonumber  \\
{\rm (v)}\ & x^t_{i,j}=0\ \text{if}\ \{i,j,i+j-2t,2D-(i+j-2t)\}\cap \{1,2,\ldots, d-1\}\neq \emptyset.\nonumber
\end{align}
\end{pro}
\begin{proof}
It is easy to see that the above constraints (i), (iii)--(v) follow directly from the definition of $x^t_{i,j}$.
We now consider constraint (ii). Let $\Phi=\{\sigma \in{\rm Aut}(X)|\textbf{0}\in \sigma C\}$. For any fixed $(i,j,t)\in \mathcal{I}$, let  $y,z\in X$ and let $(\textbf{0},y,z)\in  X^{\textbf{0}}_{i,j,t}$. Then
 by the definition of the matrix $M'$ and Proposition \ref{pro6}(i), we have that  $x_{i,j}^t=\frac{1}{|C||{\rm Aut}_{\textbf{0}}(X)|}|\{ \sigma\in \Phi|y,z\in \sigma C \}|\leq x_{i,0}^0=\frac{1}{|C||{\rm Aut}_{\textbf{0}}(X)|}|\{ \sigma\in \Phi|y\in \sigma C,\textbf{0}\in \sigma C \}|$.
\end{proof}
\subsubsection{Semidefinite programming bound on $A(\square_{2D},d)$ with even $D\geq 2$}
Based on  Proposition \ref{pro 6}, Theorem \ref{thm 2} and Proposition \ref{pro6},  the positive semidefiniteness
of $M'$  is equivalent to
\begin{align}
\text{for each even $r=0,2,\ldots,D$, the matrices\ \ \ \ \ \ \ \ \ \  \ \ \ \ \ }\nonumber\\
\bigg(\sum_{t}\beta^r_{i,j,t}x^t_{i,j}\Bigg)^{D}_{i,j=r}\label{eq22}\ \ \ \ \ \ \ \ \ \  \ \ \ \ \ \\
\text{and for each odd $r=1,3,\ldots, D-1$, the matrices\ \ \ \ \ \ \ \ \ \  \ \ \ \ \ }\nonumber\\
\Bigg(\sum_{t}\beta^r_{i,j,t}x^t_{i,j}\Bigg)^{D-1}_{i,j=r} \ \ \ \ \ \  \ \ \ \ \
\end{align}
are positive semidefinite, and
 $M''$  is equivalent to
\begin{align}
\text{for each even $r=0,2,\ldots,D$, the matrices\ \ \ \ \ \ \ \ \ \  \ \ \ \ \ }\nonumber\\
\Bigg(\sum_{t}\beta^r_{i,j,t}(x^0_{\zeta,0}-x^t_{i,j})\Bigg)^{D}_{i,j=r}\ \ \ \ \ \ \ \ \ \  \ \ \label{eq26}\\
\text{and for each odd $r=1,3,\ldots, D-1$, the matrices\ \ \ \ \ \ \ \ \ \  \ \ \ \ \ }\nonumber\\
\Bigg(\sum_{t}\beta^r_{i,j,t}(x^0_{\zeta,0}-x^t_{i,j})\Bigg)^{D-1}_{i,j=r}\label{eq25} \ \ \ \ \ \  \ \ \ \ \
\end{align}
are positive semidefinite, where $\zeta=${\rm min}$\{i+j-2t,2D-(i+j-2t)\}$.

Note that
(i) we have deleted the factors
 ${2D-2r\choose i-r}^{-\frac{1}{2}}{2D-2r\choose j-r}^{-\frac{1}{2}}$,
  $\frac{\sqrt{2}}{2}{2D-2r\choose D-r}^{-\frac{1}{2}}{2D-2r\choose j-r}^{-\frac{1}{2}}$,
$\frac{\sqrt{2}}{2}{2D-2r\choose D-r}^{-\frac{1}{2}}
                     {2D-2r\choose i-r}^{-\frac{1}{2}}$, $\frac{1}{2}{2D-2r\choose D-r}^{-1}$
as they makes the coefficients integer,
while the positive semidefiniteness is maintained; (ii) in \eqref{eq22} and \eqref{eq26}, $t\geq \lfloor\frac{j+1}{2}\rfloor$  for $i=D$ and $t\geq \lfloor\frac{i+1}{2}\rfloor$ for $j=D$.
\begin{thm}\label{thm 9} For $\square_{2D}$ with even $D\geq 2$, the semidefinite programming problem: maximize $\sum^{D-1}_{i=0}
{2D\choose i}x^0_{i,0}+\frac{1}{2}{2D\choose D}x^0_{D,0}$
subject to conditions \eqref{eq21}--\eqref{eq25} is an upper bound on $A(\square_{2D},d)$.
\end{thm}
\begin{proof}
Let  $C$ be a code with minimum distance $d$ and we view $x^t_{i,j}$ as variables. Then
$x^t_{i,j}$
 subject to conditions \eqref{eq21}--\eqref{eq25}
yields a feasible solutions with objective value $|C|$.
\end{proof}
\subsubsection{Semidefinite programming bound on $A(\square_{2D},d)$ with odd $D\geq 3$}
Based on  Proposition \ref{pro 6}, Theorem \ref{thm3} and Proposition \ref{pro6},  the positive semidefiniteness of $M'$ is equivalent to
\begin{align}
\text{for each even $r=0,2,\ldots,D-1$, the matrices\ \ \ \ \ \ \ \ \ \  \ \ \ \ \ }\nonumber\\
\bigg(\sum_{t}\beta^r_{i,j,t}x^t_{i,j}\Bigg)^{D}_{i,j=r}\label{eq15}\ \ \ \ \ \ \ \ \ \  \ \ \ \ \ \\
\text{and for each odd $r=1,3,\ldots, D-2$, the matrices\ \ \ \ \ \ \ \ \ \  \ \ \ \ \ }\nonumber\\
\Bigg(\sum_{t}\beta^r_{i,j,t}x^t_{i,j}\Bigg)^{D-1}_{i,j=r} \ \ \ \ \ \  \ \ \ \ \
\end{align}
are positive semidefinite, and
 $M''$  is equivalent to
\begin{align}
\text{for each even $r=0,2,\ldots,D-1$, the matrices\ \ \ \ \ \ \ \ \ \  \ \ \ \ \ }\nonumber\\
\Bigg(\sum_{t}\beta^r_{i,j,t}(x^0_{\zeta,0}-x^t_{i,j})\Bigg)^{D}_{i,j=r}\ \ \ \ \ \ \ \ \ \ \\
\text{and for each odd $r=1,3,\ldots, D-2$, the matrices\ \ \ \ \ \ \ \ \ \  \ \ \ \ \ }\nonumber\\
\Bigg(\sum_{t}\beta^r_{i,j,t}(x^0_{\zeta,0}-x^t_{i,j})\Bigg)^{D-1}_{i,j=r}\label{eq9} \ \ \ \ \ \  \ \ \ \ \
\end{align}
are positive semidefinite, where $\zeta=${\rm min}$\{i+j-2t,2D-(i+j-2t)\}$.
\begin{thm}\label{thm 6}
 For $\square_{2D}$ with odd $D\geq 3$, the semidefinite programming problem: maximize $\sum^{D-1}_{i=0}
{2D\choose i}x^0_{i,0}+\frac{1}{2}{2D\choose D}x^0_{D,0}$
subject to conditions \eqref{eq21} and \eqref{eq15}--\eqref{eq9} is an upper bound on $A(\square_{2D},d)$.
\end{thm}
\begin{proof}
Similar to the proof of Theorem \ref{thm 9}.
\end{proof}
\subsection{Semidefinite programming bound on $A(\square_{2D+1},d)$}
In this subsection,  we  give an  upper bound on $A(\square_{2D+1},d)$.
Given a code $C$ of $\square_{2D+1}$, for each $(i,j,t)\in \mathcal{I}'$ define the numbers
$\lambda^t_{i,j}:=|(C\times C\times C)\cap X_{i,j,t}|$
and numbers
$x^t_{i,j}:=(|C|\gamma^t_{i,j})^{-1}\lambda^t_{i,j}$,
where $\gamma^t_{i,j}$ denotes the number of nonzero entries of $M^t_{ij}$.

Recall the matrices $M'$ and $M''$ defined as in Subsection $5.1$. By the argument similar to proofs of  Propositions \ref{pro6} and \ref{pro4}, we can obtain the following propositions.
\begin{pro}\label{pro22} We have
\begin{align}\nonumber
M'=\sum_{(i,j,t)\in \mathcal{I}'}x^t_{i,j}M^t_{i,j},\ \
M''=\frac{|C|}{|X|-|C|}\sum_{(i,j,t)\in \mathcal{I}'}(x^0_{\nu,0}-x^t_{i,j})M^t_{i,j},
\end{align}
where $\nu=${\rm min}$\{i+j-2t,2D+1-(i+j-2t)\}$.
\end{pro}
\begin{pro} $x^t_{i,j},(i,j,t)\in \mathcal{I}'$ satisfy the following linear constraints,
where {\rm (v)} holds if $C$ has minimum distance at least $d$:
\begin{align}
{\rm (i)}\  &x^0_{0,0}=1.\nonumber\\
{\rm (ii)}\ & 0\leq x^t_{i,j}\leq x^0_{i,0}.\nonumber\\
{\rm (iii)}\ &\text{For}\ 0\leq i,j\leq D,\ 0\leq i+j-2t\leq D,\ x^t_{i,j}=x^{t'}_{i',j'}\  \ \label{eq23}\\
 & \text{if}\ (i',j',i'+j'-2t')\ \text{is a permutation of}\ (i,j,i+j-2t). \ \ \ \ \ \ \ \ \ \ \ \ \ \ \ \ \ \ \ \ \ \ \ \ \ \ \ \  \ \ \ \ \ \ \ \ \ \ \nonumber  \\
 {\rm (iv)}\ &\text{For}\ 0\leq i,j\leq D,\ D+1\leq i+j-2t\leq 2D,\ x^t_{i,j}=x^{t'}_{i',j'}\ \ \nonumber\\
 &\text{if}\ (i',j',2D+1-(i'+j'-2t'))\ \text{is a permutation of}\ (i,j,2D+1-(i+j-2t)). \ \ \ \ \ \ \ \ \ \ \ \ \ \ \ \ \ \ \ \ \ \ \ \ \ \ \ \  \ \ \ \ \ \ \ \ \ \ \nonumber  \\
{\rm (v)}\ &x^t_{i,j}=0\ \text{if}\ \{i,j,i+j-2t,2D+1-(i+j-2t)\}\cap \{1,2,\ldots, d-1\}\neq \emptyset.\nonumber
\end{align}
\end{pro}
Based on  Proposition \ref{pro9}, Theorem \ref{thm1} and Proposition \ref{pro22}, the positive semidefiniteness
of $M'$ and $M''$ is equivalent to
\begin{align}\label{eq28}
\text{for each $r=0,1,\ldots, D$, the matrices\ \ \ \ \ \ \ \ \ \  \ \ \ \ \ }\nonumber\\
\Bigg(\sum_{t}\beta^r_{i,j,t}x^t_{i,j}\Bigg)^{D}_{i,j=r}\\ \text{and}\ \ \Bigg(\sum_{t}\beta^r_{i,j,t}(x^0_{\nu,0}-x^t_{i,j})\Bigg)^{D}_{i,j=r}\label{eq29}
\end{align}
are positive semidefinite, where $\nu=${\rm min}$\{i+j-2t,2D+1-(i+j-2t)\}$.
\begin{thm}\label{thm 8} For $\square_{2D+1}$, the semidefinite programming problem: maximize $\sum^D_{i=0}{2D+1\choose i}x^0_{i,0}$
subject to conditions \eqref{eq23}--\eqref{eq29} is an upper bound on $A(\square_{2D+1},d)$.
\end{thm}
\begin{proof}
Similar to the proof of Theorem \ref{thm 9}.
\end{proof}
We remark that the above semidefinite programming problems in Theorems \ref{thm 9}, \ref{thm 6} and \ref{thm 8}  with $O(n^3)$ variables can be solved in time polynomial in $n$.
The obtained new bound is at least as strong as  the Delsarte's linear
programming bound \cite{d}. Indeed, diagonalizing the Bose-Mesner algebra of $\square_{n}$ yields the  Delsarte bound, which is equal to the maximum of $\sum^{\lfloor\frac{n}{2}\rfloor}_{i=0}\gamma^0_{i,0}x^0_{i,0}$ subject to the conditions $x^0_{0,0}=1$, $x^0_{1,0}=\cdots =x^0_{d-1,0}$, $x^0_{d,0}, x^0_{d+1,0},\ldots, x^0_{\lfloor\frac{n}{2}\rfloor,0}\geq 0$ and
\begin{align}\label{eq5}
\sum^{\lfloor\frac{n}{2}\rfloor}_{i=0}x^0_{i,0}A_i\ \  \text{is positive semidefinite,}
\end{align}
 where $A_i$ is the $i$th distance matrix of $\square_{n}$. Note that condition  \eqref{eq5} can be  implied by the condition
 that $M'$ and $M''$ is positive semidefinite.

\subsection{Computational results}
In this subsection we  give, in the range $8\leq n\leq 13$, several concrete semidefinite programming  bounds  and  Delsarte's linear programming bounds on $A(\square_{n},d)$, respectively. The latter involves the  second eigenmatrix of $\square_n$.
\begin{lem} Let $\bar{q}_j(i)\ (0\leq i, j\leq \lfloor \frac{n}{2}\rfloor)$ be the $(i,j)$-entry of this eigenmatrix. Then we have
$\bar{q}_j(i)=\sum^{2j}_{k=0}(-1)^k{i\choose k}{n-i\choose 2j-k}$.
\end{lem}
\begin{proof}
We first recall the following fact. Let $\Gamma$ denote a distance-regular graph with diameter $D$ and intersection numbers $c_i,a_i,b_i\ (0\leq i\leq D)$. Without loss of generality, we assume its eigenvalues $\theta_0>\theta_1>\cdots >\theta_D$. Let $q_j(i)\ (0\leq i, j\leq D)$ be the $(i,j)$-entry of the  second eigenmatrix of $\Gamma$.
Then we have $c_iq_j(i-1)+a_iq_j(i)+b_iq_j(i+1)=\theta_jq_j(i)$ ($0\leq j\leq D$) by \cite[p. 128]{bcn}.

When $\Gamma$ is $H(n,2)$, it is known that $q_{j}(i)=\sum^{j}_{k=0}(-1)^k{i\choose k}{n-i\choose j-k}\ (0\leq i, j\leq n)$ is the  $(i,j)$-entry of the second  eigenmatrix of $H(n,2)$. Then by comparing   the above identity for $H(n,2)$ with that for $\square_n$, one can easily finds that $\bar{q}_j(i)=q_{2j}(i)\ (0\leq i, j\leq \lfloor \frac{n}{2}\rfloor)$.
\end{proof}

The followings are our computational resuls.\\ [.2cm]
New upper bounds on $A(\square_{2D},d)$\text{\ \ \ \ \ \  \ \ \ \ \ \ \ \ \ \  \ \ \ \ \ \ \ \ \ } New upper bounds on $A(\square_{2D+1},d)$\\
\begin{tabular}{cccc}
\hline
{$ D$}                    &{$d$}  &$\stackrel{\rm  New\ upper}{\rm bound}$ &\text{Delsarte bound} \\\hline
{$4$}   &$2$        &{28}                 &$64$         \\
{$5$}   &$2$                 &{256}                 &$256$\\\hline
{$5$}   &$3$                 &{24}                 &$32$\\
{$6$}   &$3$                 &{87}                 &$128$\\ \hline
{$5$}   &$4$                 &{16}                 &$16$\\
{$6$}   &$4$                 &{54}                 &$85$\\  \hline
\end{tabular}
\ \ \ \ \  \ \ \ \ \ \ \ \ \ \  \ \ \ \ \  \ \ \
\begin{tabular}{cccc}
\hline
{$ D$}                    &{$d$}  &$\stackrel{\rm  New\ upper}{\rm bound}$ &\text{Delsarte bound} \\\hline
{$4$}   &$2$        &{93}                 &$112$         \\
{$6$}   &$2$        &{1348}                 &$1877$         \\ \hline
{$5$}   &$3$                 &{85}                 &$85$\\
{$6$}   &$3$                 &{213}                 &$213$\\ \hline
{$5$}   &$4$                 &{20}                 &$27$\\
{$6$}   &$4$                 &{111}                 &$120$\\ \hline
\end{tabular}

\section*{Acknowledgement}
This work was supported by the NSF of China (No. 11471097) and the NSF of Hebei Province (No. A2017403010).


\begin{thebibliography}{00}


\bibitem{ban} E. Bannai,  T.Ito,  Algebraic Combinatorics I: Association Schemes,
The Benjamin-Cummings Lecture Notes Ser. vol. 58, Benjamin-Cummings, Menlo Park, 1984.

\bibitem{bcn} A.E. Brouwer,   A.M. Cohen, A. Neumaier,  Distance-Regular Graphs,  Springer, Berlin, 1989.
\bibitem{cau}  J.S. Caughman, The  Terwilliger algebra of
bipartite $P$- and $Q$-Polynomial schemes, Discrete Math. 196 (1999) 65--95.

\bibitem{ca}  J.S. Caughman, M.S. MacLean, P.Terwilliger,  The  Terwilliger algebra of
almost bipartite $P$- and $Q$-Polynomial schemes, Discrete Math. 292 (2005) 17--44.

\bibitem{d} P. Delsarte,  An algebraic approach to the association schemes of coding
theory, Philips Res. Repts. Suppl., no. 10, 1973.

\bibitem{g} D. Gijswijt,  Matrix algebras and semidefinite programming techniques for
codes,  Ph.D. thesis, arxiv.1007.0906.

\bibitem{g1} D. Gijswijt, A. Schrijver, H. Tanaka,  New upper bounds for nonbinary codes, J. Combin.
Theory Ser. A 13 (2006) 1719--1731.

 \bibitem{s} A. Schrijver, New code upper bounds from the Terwilliger
 algebra and semidefinite programming, IEEE Trans. Inform. Theory 51 (2005) 2859--2866.

\bibitem{p2} P. Terwilliger,  The subconstituent algebra of an association scheme I, J. Algebraic
Combin.  1 (1992) 363--388.

\bibitem{ter3} P. Terwilliger, The subconstituent algebra of an association scheme III, J.  Algebraic Combin.   2 (1993) 177--210.

\end{thebibliography}
\end{document}